\documentclass[fleqn, eqnumis, pdf]{hjms}

\usepackage{footnote}
\usepackage{multicol}
\usepackage{booktabs}
\usepackage[flushleft]{threeparttable}
\usepackage[font=small,labelfont=bf]{caption}

\begin{document}

\title{An Explicit Formula of the Intrinsic Metric on the Sierpinski Gasket via Code Representation}

\author{
Mustafa Saltan\footnote{Department of Mathematics, Anadolu University,  Turkey, \newline \ E-mail:{\tt mustafasaltan@anadolu.edu.tr}},
Yunus Ozdemir\footnote{Department of Mathematics, Anadolu University,  Turkey,
\ E-mail: {\tt  yunuso@anadolu.edu.tr}}\footnote{Corresponding Author.}
and
Bunyamin Demir\footnote{Department of Mathematics, Anadolu University,  Turkey,
\ E-mail: {\tt  bdemir@anadolu.edu.tr}}
}
{ }

\begin{abstract}
{The computation of the distance between any two points of the Sierpinski Gasket with respect to the intrinsic metric has already been investigated by several authors. In the literature there is not an explicit formula using the code space of the Sierpinski Gasket. In this paper, we give an explicit formula for the intrinsic metric on the Sierpinski Gasket via code representations of its points.}
\end{abstract}



\section{Introduction}

The Sierpinski Gasket was described by W. Sierpinski in 1915, and then it became one of the typical examples of fractals.  Sierpinski Gasket has been studied in fractal geometry for years (see for example \cite{Barnsley,Kigami} for more information). It is well known that $S$ is the attractor of the iterated function system $\{\mathbb{R}^{2};f_{0},f_{1},f_{2}\}$ where
\begin{eqnarray*}
f_{0}(x,y)&=&\left(\frac{1}{2}x,\frac{1}{2}y\right)\\
f_{1}(x,y)&=&\left(\frac{1}{2}x+\frac{1}{2},\frac{1}{2}y\right)\\
f_{2}(x,y)&=&\left(\frac{1}{2}x+\frac{1}{4},\frac{1}{2}y+\frac{\sqrt{3}}{4}\right).
\end{eqnarray*}

\begin{figure}[h]
\centering\includegraphics[scale=0.7]{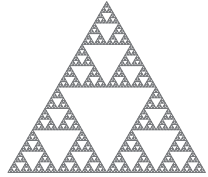}
\caption{The Sierpinski Gasket as an attractor of an IFS.}\label{sier}
\end{figure}

In \cite{Gabner, Hinz}, the authors define $S$ as follows:
Let $P_{0}=(0,0), P_{1}=(0,1)$ and $P_{2}=(\frac{1}{2},\frac{\sqrt{3}}{2})$. Assume that $i_{1}i_{2}\ldots i_{n}$ is the word of length $n$ over
the alphabet $X=\{0, 1, 2\}$ for any $i_{1},i_{2},\ldots, i_{n}\in X$. For every such word, it is denoted the elementary sub-triangle of level $n$ with vertices $f_{i_{1}}(P_{0})\circ f_{i_{2}}(P_{0})\circ\ldots \circ f_{i_{n}}(P_{0})$,  f$_{i_{1}}(P_{1})\circ f_{i_{2}}(P_{1})\circ\ldots \circ f_{i_{n}}(P_{1})$ and $f_{i_{1}}(P_{2})\circ f_{i_{2}}(P_{2})\circ\ldots \circ f_{i_{n}}(P_{2})$ by $T_{i_{1} i_{2} \ldots  i_{n}}$. Then they define the Sierpinski Gasket as \[S=\bigcup_{n\geq 0}T_{n} \  \ \text{where}\ \ T_{n}=\bigcup_{s\in \{0,1,2\}^{n}}T_{s}.\]

As known, it can be constructed several metric structures on a set. But, a metric which is not take into consideration its internal structure is far from being applicable. For example, consider the restriction of Euclidean metric to $S$. According to this metric, the distance  between $a$ and $b$ is $l$ (see Figure~\ref{s0}). However, there is not any path between $a$ and $b$ on $S$ with length $l$. For this reason, this metric is not meaningful on this special set. The intrinsic metric which is obtained by taking into account the paths on the structure, eliminates this discrepancy.

\begin{figure}[h]
\centering\includegraphics[scale=0.7]{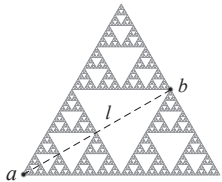}
\caption{Distance between two points on $S$ with respect to the Euclidean metric.}\label{s0}
\end{figure}

One can define the intrinsic metric on $S$ such that
\[d(x, y)=\inf \{\delta \ | \  \delta \ \text{is the length of a rectifiable curve in S joining} \ x \  \text{and} \ y \}\]
for $x, y \in S$ (for details see \cite{Burago}).

In the several works, the intrinsic metric on the Sierpinski Gasket was constructed and defined in different ways since there exist different ways to construct (or define) the Sierpinski Gasket (for details see \cite{Cristea, Gabner, Hinz, Kigami}). For example in \cite{Gabner}, it is given an alternative definition of the intrinsic metric on $S$ as follows: Let $x,y\in S$ and let $\Delta_{n}(x), \Delta_{n}(y) $ be two elementary sub-triangles of level $n$ where $x\in \Delta_{n}(x)$ and $y\in \Delta_{n}(y)$ for all $n\geq 0$. For every $n\geq 0$, the left lower vertices of $\Delta_{n}(x)$ and $ \Delta_{n}(y)$ respectively. Then the authors define the intrinic metric as
\[d(x, y)=\lim_{n\rightarrow \infty} \frac{d_{n}(x_{n},y_{n})}{2^{n}}\]
where $x,y\in S$.

R. Strichartz also defines the intrinsic metric in a different way by using barycentric coordinates (for details see \cite{Str}).

In \cite{Romik}, Romik tackle the discrete Sierpinski Gasket and define the metric giving the shortest distance on the points of this set using by the code spaces. Romik then compute the average distance between points on the Sierpinski Gasket using the connection between the Tower of Hanoi problem and the discrete Sierpinski Gasket.

In this paper, we use code representations of the points of the Sierpinski Gasket to define the intrinsic metric. We note that the junction points of the Sierpinski Gasket have two different code representations.
In this work, we give an explicit formula for the intrinsic metric on $S$ such that the formula does not depend on the choice of the representations of the junction points as mentioned in Proposition~\ref{temsilcidenbagimsiz}.

\bigskip

\section{Code representation on the Sierpinski Gasket}
We first give a small brief about the coding process.

Let us denote the left-bottom part, the right-bottom part and the upper part of the Sierpinski Gasket by $S_{0}, S_{1}$ and $S_{2}$ respectively (see Figure \ref{s1}).
\begin{figure}[ht]
\centering\includegraphics[scale=0.8]{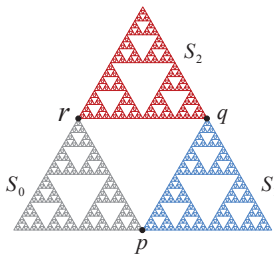}
\caption{The sub-triangles $S_{0}, S_{1}$ and $S_{2}$ of $S$.}\label{s1}
\end{figure}

As shown in Figure~\ref{s1}, $S=S_{0}\cup S_{1}\cup S_{2}$, $S_{0}\cap S_{1}=\{p\}$, $S_{1}\cap S_{2}=\{q\}$ and $S_{0}\cap S_{2}=\{r\}$. Let $a_{1}\in \{0,1,2\}$. Now similarly we denote the left-bottom part, the right-bottom part and the upper part of $S_{a_{1}}$ by $\displaystyle  S_{a_{1}0}, S_{a_{1}1}$ and $S_{a_{1}2}$ respectively (see Figure \ref{s2}).
\begin{figure}[ht]
\centering\includegraphics[scale=0.8]{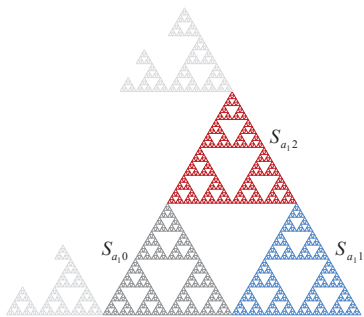}
\caption{The sub-triangles $S_{a_{1}0}, S_{a_{1}1}$ and $S_{a_{1}2}$ of $S_{a_{1}}$ for $a_{1}=1$.}\label{s2}
\end{figure}

With the same argument, let $S_{a_{1}a_{2}\ldots a_{k}}$ denote the smaller triangular pieces of $S$ where $a_{i}\in \{0,1,2\}$ and $i=1,2,\ldots, k.$ For the sequence
\[S_{a_{1}}, S_{a_{1}a_{2}}, S_{a_{1}a_{2}a_{3}},\ldots, S_{a_{1}a_{2}\ldots a_{n}},\ldots,\]
it is obvious that $S_{a_{1}}\supset S_{a_{1}a_{2}}\supset S_{a_{1}a_{2}a_{3}}\supset \ldots \supset S_{a_{1}a_{2} \ldots a_{n}} \supset\ldots $
and the infinite intersection
\[\bigcap_{k=1}^{\infty}S_{a_{1}a_{2} \ldots a_{k}}\]
is a singleton, say $\{a\}$ where $a\in S$. We denote the point $a\in S$  by $a_{1}a_{2}\ldots a_{n}\ldots $ where $a_{n}\in \{0,1,2\}$ and $n=1,2,\ldots $. Note that, if $a\in S$ is the intersection point of any two sub-triangles of $S_{a_{1}a_{2} \ldots a_{k}}$ (such a point is called a junction point of $S$) then $a$ has two different representations such that $a_{1}a_{2} \ldots a_{k}\beta\alpha\alpha\alpha\alpha\ldots$ and $a_{1}a_{2} \ldots a_{k}\alpha\beta\beta\beta\beta\ldots$ where $\alpha,\beta \in \{0,1,2\}$ (see Figure \ref{s3}). Otherwise, $a$ has a unique representation.

(For an alternative code space representation of the points of $S$, see \cite{Deniz}.)

\begin{figure}[h]
\centering\includegraphics[scale=0.8]{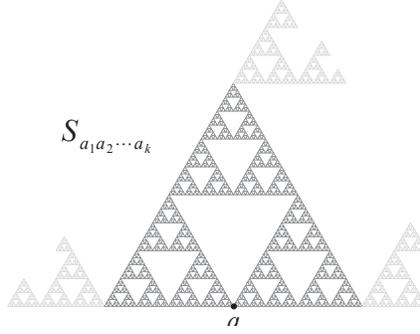}
\caption{The point $a$ for $\alpha=0$ and $\beta=1$.}\label{s3}
\end{figure}

\section{Construction of the intrinsic metric on $S$}
Let $a$ and $b$ be two different points of $S$ whose representations are $a=a_{1}a_{2} \ldots a_{n}\ldots$ and $b=b_{1}b_{2} \ldots b_{n}\ldots$ respectively. Then there exists a natural number $s$ such that $a_{s}\neq b_{s}$. Let \begin{equation}\label{find-k}
k=\min\{s \ | \ a_{s}\neq b_{s}, \ s=1,2,3,\ldots \}.
\end{equation}
We then have $a\in S_{a_{1}a_{2} \ldots a_{k-1}a_{k}}$ and $b\in S_{a_{1}a_{2} \ldots a_{k-1}b_{k}}$. Without lost of generality, we assume that $a_{k}=0$ and $b_{k}=1$ which means $a\in S_{a_{1}a_{2} \ldots a_{k-1}0}$ and $b\in S_{a_{1}a_{2} \ldots a_{k-1}1}$ as seen in Figure~\ref{s4} (we use the abbreviation $\sigma=a_{1}a_{2} \ldots a_{k-1}$ for simplicity). Note also that, in the other cases, i.e. $a$ and $b$ are in another sub-triangle of $S_{a_{1}a_{2} \ldots a_{k-1}}$, similar procedures would be valid.

\begin{figure}[h]
\centering\includegraphics[scale=0.7]{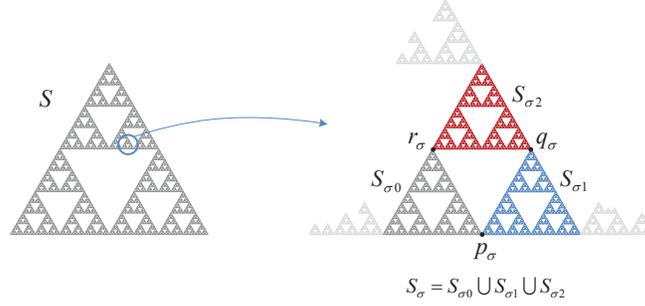}
\caption{The sub-triangle $S_{\sigma}$ where $\sigma=a_{1}a_{2} \ldots a_{k-1}$ and the points $a\in S_{\sigma 0}$ and $b\in S_{\sigma 1}$.}\label{s4}
\end{figure}

Let $p_{\sigma}$, $r_{\sigma}$, $q_{\sigma}$  be the intersection points of the sub-triangles $S_{\sigma0}$ and $S_{\sigma1}$, $S_{\sigma0}$ and $S_{\sigma2}$, $S_{\sigma1}$ and $S_{\sigma2}$ respectively. The shortest paths between $a$ and $b$  must pass through either the point $p_{\sigma}$ or the line $r_{\sigma}q_{\sigma}$ (see Figure \ref{s4}).

We now investigate these two different ways as follows:

\noindent\textbf{Case 1:} First consider the shortest path passing through the point $p_{\sigma}$.  Any path between $a$ and $b$ can be expressed as the union of a path between $a$ and $p_{\sigma}$ and a path between $p_{\sigma}$ and $b$. We first look at the shortest paths between $a$ and $p_{\sigma}$ (The paths between $p_{\sigma}$ and $b$ can be obtained using similar argument).

$\bullet$ If $a\in S_{a_{1}a_{2} \ldots a_{k-1}00}$ or $a\in S_{a_{1}a_{2} \ldots a_{k-1}02}$  then we must compute the length of the line segment $p_{\sigma'} p_{\sigma}$ or the length of the line segment $q_{\sigma'} p_{\sigma}$ where $p_{\sigma'}$,$q_{\sigma'}$ are the intersection points of the sub-triangles $S_{\sigma'0}$ and $S_{\sigma'1}$, $S_{\sigma'1}$ and $S_{\sigma'2}$ respectively where $\sigma'=a_{1}a_{2} \ldots a_{k-1}0$(see Figure \ref{s5}).
In the both cases, the length of the shortest paths between $a$ and $p_{\sigma}$ is
\[\displaystyle \mu =\frac{1}{2^{k+1}}+\varepsilon,\]
for some $\varepsilon\geq 0$.

For the case $a=r_{\sigma'}$, where $r_{\sigma'}$ is the intersection point of the sub-triangles $S_{\sigma'0}$ and $S_{\sigma'2}$, there exist obviously two shortest paths between $a$ and $p_{\sigma}$ (see Figure~\ref{s5}). These paths are the union of the line segments $r_{\sigma'} p_{\sigma'}$ and $p_{\sigma'} p_{\sigma}$ or the union of the line segments $r_{\sigma'} q_{\sigma'}$ and $q_{\sigma'} p_{\sigma}$. The length of these paths can be easily computed as $\mu =\dfrac{1}{2^{k}}.$

\begin{figure}[h]
\centering\includegraphics[scale=0.7]{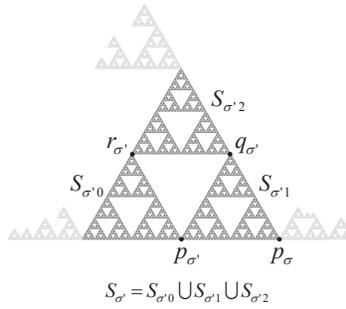}
\caption{The sub-triangle $S_{\sigma'}$ where $\sigma'=a_{1}a_{2} \ldots a_{k-1}0$ and the points $a\in S_{\sigma' 0}$ or $b\in S_{\sigma' 2}$.}\label{s5}
\end{figure}

$\bullet$ Suppose that $a\in S_{a_{1}a_{2} \ldots a_{k-1}01}$.

If $a\in S_{a_{1}a_{2} \ldots a_{k-1}010}$ or $a\in S_{a_{1}a_{2} \ldots a_{k-1}012}$, then we must compute the length of the line segment $p_{\sigma''}p_{\sigma}$ or the length of the line segment $q_{\sigma''}p_{\sigma}$ where $p_{\sigma''}$,$q_{\sigma''}$ are the intersection points of the sub-triangles $S_{\sigma''0}$ and $S_{\sigma''1}$, $S_{\sigma''1}$ and $S_{\sigma''2}$ respectively where $\sigma''=a_{1}a_{2} \ldots a_{k-1}01$(see Figure \ref{s6}).

In the both cases, we get  $$\mu=\frac{1}{2^{k+2}}+\varepsilon,$$
for some $\varepsilon\geq 0 $.

For the case $a=r_{\sigma''}$, where $r_{\sigma''}$ is the intersection point of the sub-triangles $S_{\sigma''0}$ and $S_{\sigma''2}$, there are two paths giving  the distance of the shortest paths between $a$ and $p_{\sigma}$ as were before. These paths are the union of the line segments $r_{\sigma''}p_{\sigma''}$ and $p_{\sigma''}p_{\sigma}$ or the union of the line segments $r_{\sigma''}q_{\sigma''}$ and $q_{\sigma''}p_{\sigma}$. The length of these two paths is $\mu =\frac{1}{2^{k+1}}$.

\begin{figure}[h]
\centering\includegraphics[scale=0.7]{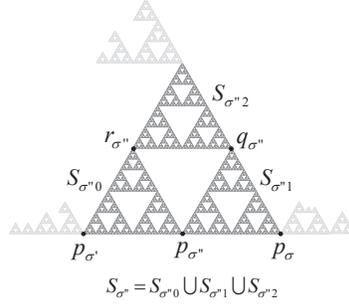}
\caption{The sub-triangle $S_{\sigma''}$ where $\sigma''=a_{1}a_{2} \ldots a_{k-1} \, 01$.}\label{s6}
\end{figure}

Using similar procedure for smaller triangles, we can determine the shortest paths between $a$ and $b$ and the length of these paths. Similarly one can determine the shortest paths between $p_{\sigma}$ and $b$. Then by splicing these shortest paths, between ``$a$ and  $p_{\sigma}$'' and ``$p_{\sigma}$ and $b$'', one can compute the length of the shortest paths between $a$ and $b$ passing through the point $p_{\sigma}$.

\noindent\textbf{Case 2:} Let us consider the shortest paths passing through the line segment $r_{\sigma}q_{\sigma}$. In a similar way, we can obtain the shortest paths (thus the corresponding length) between ``$a$ and $r_{\sigma}$'' and between ``$b$ and $q_{\sigma}$''. As we add $\frac{1}{2^{k}}$ (that is, the length of the path $r_{\sigma}q_{\sigma}$) to these length, we obtain the length of the shortest path passing through $r_\sigma q_\sigma$.

Consequently, the length of the shortest paths between $a$ and $b$ is the minimum of the lengths obtained from Case 1 and Case 2. We can formulate this length  (so the metric) as follows:

\bigskip

\begin{definition}\label{def1}
Let $a_{1}a_{2} \ldots a_{k-1}a_{k}a_{k+1} \ldots$ and $b_{1}b_{2} \ldots b_{k-1}b_{k}b_{k+1}\ldots $ be two representations respectively of the points $a\in S$ and $b\in S$ such that $a_{i}=b_{i}$ for $i=1,2,\ldots, k-1$ and $a_{k}\neq b_{k}.$ We define the distance $d(a,b)$ between $a$ and  $b$ as
\[d(a,b)=\min  \left\{\sum\limits_{i=k+1}^{\infty }\frac{\alpha_{i}+\beta_{i}}{2^{i}}\, , \, \frac{1}{2^{k}}+\sum\limits_{i=k+1}^{\infty }\frac{\gamma_{i}+\delta_{i}}{2^{i}} \right\}\]
where

\vspace{0.2cm}
$ \alpha _{i}= \left\{
\begin{array}{cc}
0, &  a_{i}=b_{k} \\
1, &  a_{i}\neq b_{k}
\end{array}%
\right. ,$

\vspace{0.2cm}

$ \beta _{i} =\left\{
\begin{array}{cc}
0, &  b_{i}=a_{k} \\
1, &  b_{i}\neq a_{k}
\end{array}
\right. ,$

\vspace{0.2cm}

$ \gamma _{i}= \left\{
\begin{array}{cl}
 0, &  a_{i}\neq a_{k}~\ \text{and} \text{ }a_{i}\neq b_{k} \\
 1, &  otherwise
\end{array}
\right. ,$

\vspace{0.2cm}

$ \delta _{i}=\left\{
\begin{array}{cl}
0, & b_{i}\neq b_{k}\text{ }\ \text{and} \text{ }b_{i}\neq a_{k} \\
1, &  otherwise
\end{array}
\right. . $
\end{definition}
\bigskip

\begin{remark}
Note that the first value $\sum\limits_{i=k+1}^{\infty }\frac{\alpha_{i}+\beta_{i}}{2^{i}}$ is the length of the shortest paths passing through the point $p_\sigma$
and the second value $\frac{1}{2^{k}}+\sum\limits_{i=k+1}^{\infty }\frac{\gamma_{i}+\delta_{i}}{2^{i}} $ is the length of the shortest paths passing through the line segment $r_\sigma q_\sigma$ where $\frac{1}{2^{k}}$ is the length of the line segment $r_\sigma q_\sigma$.
\end{remark}

\bigskip

\begin{proposition}
The distance function $d$ defined in Definition~\ref{def1} is strictly intrinsic metric on $S$.
\end{proposition}
\begin{proof}
It is obvious from the fact that $d(a,b)$ is defined as the minimum of the lengths of the admissible paths connecting the points $a$ and $b$ in $S$.
\end{proof}

\begin{proposition}\label{temsilcidenbagimsiz}
The metric $d$ defined in Definition~\ref{def1} does not depend on the choice of the code representations of the points.
\end{proposition}

\begin{proof}
Let $a$ be a junction point whose code representations are of the form $a_{1}a_{2}a_{2} \ldots a_{2}a_{2}a_{2}\ldots$ and $a_{2}a_{1} a_{1} \ldots a_{1}a_{1}a_{1}\ldots$ such that $a_1 \neq a_2$ (in the general case, i.e. if the code representation of $a$ is of the form $a_{1}a_{2} \ldots a_{k-1}a_{k}a_{k+1} a_{k+1}a_{k+1}\ldots$, the claim can be proven similarly).

Let $x$ be an arbitrary point of $S$ which has the code representation \[x_{1}x_{2} \ldots x_{k-1}x_{k}x_{k+1} x_{k+2}x_{k+3}\ldots.\]

Assume that $x_{1}\neq a_{1}$. In this case, it must be $x_{1}\neq a_{2}$ or $x_{1}=a_{2}$.

\textbf{Case 1:} We first take $x_{1}\neq a_{2}$. We now investigate the distance between the points $$x_{1}x_{2} \ldots x_{k}x_{k+1} x_{k+2}x_{k+3}\ldots \ \text{and} \  a_{1}a_{2}a_{2} \ldots a_{2}a_{2}a_{2}\ldots.$$ Due to the definition of $d$, we have the following equations:

$ \alpha _{i}= \left\{
\begin{array}{cc}
0, &  x_{i}=a_{1} \\
1, &  x_{i}\neq a_{1}
\end{array}%
\right. ,$

\vspace{0.2cm}

$ \beta _{i} =\left\{
\begin{array}{cc}
0, &  a_{2}=x_{1} \\
1, &  a_{2}\neq x_{1}
\end{array}
\right. ,$

\vspace{0.2cm}

$ \gamma _{i}= \left\{
\begin{array}{cl}
 0, &  x_{i}\neq x_{1}~\ \text{and} \text{ }x_{i}\neq a_{1} \\
 1, &  otherwise
\end{array}
\right. ,$

\vspace{0.2cm}

$ \delta _{i}=\left\{
\begin{array}{cl}
0, & a_{2}\neq a_{1}\text{ }\ \text{and} \text{ }a_{2}\neq x_{1} \\
1, &  otherwise
\end{array}
\right. . $

\vspace{0.2cm}

We thus get $\beta _{i}=1$ for all $i\geq 2$ owing to the fact that  $x_{1}\neq a_{2}$. Moreover, $\alpha _{i}$ can change according to the value of  $x_{i}$ and $a_{1}$ for each $i\geq 2$. 
 It is also easily seen that $\delta _{i}=0$ for every $i\geq 2$ since $a_{2}\neq a_{1} $ and $a_{2}\neq x_{1} $. It follows that
 \[\sum\limits_{i=2}^{\infty }\frac{\alpha_{i}+\beta_{i}}{2^{i}}=\frac{1}{2}+\sum\limits_{i=2}^{\infty }\frac{\alpha_{i}}{2^{i}}\]
  and
  \[\frac{1}{2}+\sum\limits_{i=2}^{\infty }\frac{\gamma_{i}+\delta_{i}}{2^{i}}=\frac{1}{2}+\sum\limits_{i=2}^{\infty }\frac{\gamma_{i}}{2^{i}}.\]



Now we compute the distance between the points $$x_{1}x_{2} \ldots x_{k}x_{k+1} x_{k+2}x_{k+3}\ldots \  \text{and} \  a_{2}a_{1}a_{1} \ldots a_{1}a_{1}a_{1}\ldots.$$
Thanks to the definition of $d$, we have the following equations:

$ \alpha _{i}'= \left\{
\begin{array}{cc}
0, &  x_{i}=a_{2} \\
1, &  x_{i}\neq a_{2}
\end{array}%
\right. ,$

\vspace{0.2cm}

$ \beta _{i}' =\left\{
\begin{array}{cc}
0, &  a_{1}=x_{1} \\
1, &  a_{1}\neq x_{1}
\end{array}
\right. ,$

\vspace{0.2cm}

$ \gamma _{i}'= \left\{
\begin{array}{cl}
 0, &  x_{i}\neq x_{1}~\ \text{and} \text{ }x_{i}\neq a_{2} \\
 1, &  otherwise
\end{array}
\right. ,$

\vspace{0.2cm}

$ \delta _{i}'=\left\{
\begin{array}{cl}
0, & a_{2}\neq a_{1}\text{ }\ \text{and} \text{ }a_{1}\neq x_{1} \\
1, &  otherwise
\end{array}
\right. . $

\vspace{0.2cm}

Similarly, we have $\beta _{i}'=1$ for all $i\geq 2$ owing to the fact that  $x_{1}\neq a_{1}$. Moreover, $\alpha _{i}'$ can change according to the value of  $x_{i}$ and $a_{2}$ for each $i\geq 2$.
It is also obviously seen that  $\delta _{i}'=0$ for every $i\geq 2$ since $a_{1}\neq a_{2} $ and $a_{1}\neq x_{1} $. This shows that
\[\sum\limits_{i=2}^{\infty }\frac{\alpha_{i}'+\beta_{i}'}{2^{i}}=\frac{1}{2}+\sum\limits_{i=2}^{\infty }\frac{\alpha_{i}'}{2^{i}}\]
and
\[\frac{1}{2}+\sum\limits_{i=2}^{\infty }\frac{\gamma_{i}'+\delta_{i}'}{2^{i}}=\frac{1}{2}+\sum\limits_{i=2}^{\infty }\frac{\gamma_{i}'}{2^{i}}.\]





Finally we show that $\alpha_{i}=\gamma_{i}'$ and $\alpha_{i}'=\gamma_{i}$ for all $i\geq 2$ respectively. We have already known that $a_{1}\neq a_{2}$, $x_{1}\neq a_{1}$ and $x_{1}\neq a_{2}$.

Assume that  $\gamma_{i}'=0$ for a fixed $i$. In this case, we have $x_{i}\neq a_{2}$ and $x_{i}\neq x_{1}$. We thus have $x_{i}= a_{1}$. Namely, it is $\alpha_{i}=0$. Let $\gamma_{i}'=1$ for a fixed $i$. Hence it must be $x_{i}= a_{2}$ or $x_{i}= x_{1}$. This shows that $x_{i}\neq a_{1}$. That is we obtain $\alpha_{i}=1.$

Suppose that $\gamma_{i}=0$ for a fixed $i$. We thus have $x_{i}\neq x_{1}$ and $x_{i}\neq a_{1}$ and this shows that $x_{i}= a_{2}$. So we get $\alpha_{i}'=0$. Let $\gamma_{i}'=1$ for a fixed $i$. Therefore it must be $x_{i}= x_{1}$ or $x_{i}= a_{1}$. It follows that $x_{i}\neq a_{1}$ and thus we get $\alpha_{i}'=1.$

This concludes the proof in Case 1.

\textbf{Case 2:} Let $x_{1}= a_{2}$. The assertion can be proved similarly as above.
\end{proof}

\section{Some Examples}
In this section we give two examples in which we compute the distance between two kinds of pair of points in $S$.

\begin{example}
Let $a$ and $b$ be the points in $S$ whose representations are $\overline{012}=012012012\cdots$ and $\overline{1}=111\cdots$ respectively (see Figure~\ref{figex1} for the place of the points).

\begin{figure}[h!]
\centering\includegraphics[scale=0.7]{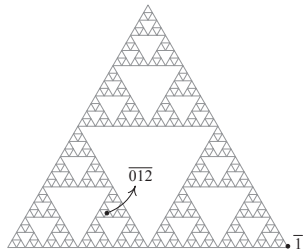}
\caption{The points $a$ and $b$ coded by $\overline{012}$ and $\overline{1}$ respectively.}\label{figex1}
\end{figure}

To compute $d(a,b)$ we need the natural number $k$ defined in (\ref{find-k}). Since the first term of the representations are different, we get $k=1$. Easy calculations give us $\beta_i=1$, $\delta_i=1$,
\[
\alpha_i=
\left\{
  \begin{array}{ccc}
    0 &  ; &  $i$ \ \equiv \ $2$ \ (mod \, 3) \\
    1 &  ; & otherwise\\
  \end{array}
\right.
\]
and
\[
\gamma_i=
\left\{
  \begin{array}{ccc}
    0 &  ; &  $i$ \ \equiv \ $0$ \ (mod \, 3) \\
    1 &  ; & otherwise\\
  \end{array}
\right.
\]
for all $i\geq k+1=2$. We then obtain
\[
\sum\limits_{i=2}^{\infty }\frac{\alpha_{i}+\beta_{i}}{2^{i}}
=\sum\limits_{m=1}^{\infty }\left(  \frac{1}{2^{3m-1}} + \frac{2}{2^{3m}}+ \frac{2}{2^{3m+1}}\right)=\frac{5}{7}
\]
and
\[
\frac{1}{2}+\sum\limits_{i=2}^{\infty }\frac{\gamma_{i}+\delta_{i}}{2^{i}}
=\frac{1}{2}+\sum\limits_{m=1}^{\infty }\left(  \frac{2}{2^{3m-1}} + \frac{1}{2^{3m}}+ \frac{2}{2^{3m+1}}\right)=\frac{1}{2}+\frac{6}{7}
\]
which says that $d(a,b)$ is the minimum value $\displaystyle \frac{5}{7}$.
\end{example}

\begin{example}
Let $a$ and $b$ be the points in $S$ whose representations are $000\overline{2}=000222222\cdots$ and $0122\overline{0}=0122000000\cdots$ respectively (see Figure~\ref{figex2} for the place of the points).

\begin{figure}[h!]
\centering\includegraphics[scale=0.7]{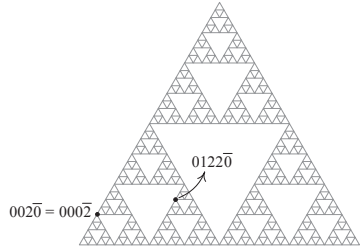}
\caption{The points $a$ and $b$ coded by $000\overline{2}$ and $0122\overline{0}$ respectively.}\label{figex2}
\end{figure}

Since the second term of the representations are different, we get $k=2$. One can obtain $\alpha_i=1$ for $i\geq k+1=3$, $\beta_3=\beta_4=1$ and $\beta_i=0$ for $i\geq 5$,  $\gamma_3=1$ and $\gamma_i=0$ for $i\geq 4$, $\delta_3=\delta_4=0$ and $\delta_i=1$ for $i\geq 5$. We then obtain
\[
\sum\limits_{i=3}^{\infty }\frac{\alpha_{i}+\beta_{i}}{2^{i}}
=\frac{2}{2^3}+ \frac{2}{2^4} + \sum\limits_{i=5}^{\infty }  \frac{1}{2^i}=\frac{7}{16}
\]
and
\[
\frac{1}{2^2}+\sum\limits_{i=3}^{\infty }\frac{\gamma_{i}+\delta_{i}}{2^{i}}
=\frac{1}{4}+\frac{1}{2^3}+ \sum\limits_{i=5}^{\infty }  \frac{1}{2^i}=\frac{7}{16}
\]
which says that $d(a,b)$ is the value $\displaystyle \frac{7}{16}$. Notice that two values are equal and it means that there exist at least two shortest paths between the points.

Indeed, since it is a junction point, the point $000\overline{2}$ has two code representations and one can take the representation of this point as $002\overline{0}$. In this case the computation yields
$k=2$, $\alpha_i=1$ for $i\geq 3$, $\beta_3=\beta_4=1$ and $\beta_i=0$ for $i\geq 5$,  $\gamma_3=0$ and $\gamma_i=1$ for $i\geq 4$, $\delta_3=\delta_4=0$ and $\delta_i=1$ for $i\geq 5$. We then get by easy calculation $d(a,b)=\displaystyle \frac{7}{16}$ again as mentioned in Proposition~\ref{temsilcidenbagimsiz}.
\end{example}


\end{document}